\documentclass[11pt]{amsart}
\usepackage[utf8]{inputenc}

\usepackage{amssymb,amsthm,amsfonts,enumerate, epsfig,epstopdf,caption, graphicx, xcolor}
\usepackage{amsmath}
\usepackage{subcaption}
\usepackage{hyperref}

\newcommand{\lastcfrac}[2]{%
  \vphantom{\cfrac{#1}{#2}}%
  \raisebox{\dimexpr1ex-\height}{%
    $\displaystyle \raisebox{.5\height}{$\ddots$}+\cfrac{#1}{#2} $%
  }%
}

\title{Constructions of $q$-hyperbolic knots}
\author{Efstratia Kalfagianni and Joseph M. Melby}
\thanks{ This research is   partially supported by the NSF  grants  DMS-2004155 and DMS-2304033.}

\address[]{Department of Mathematics, Michigan State University, East Lansing, MI, 48824, USA}
\email[]{kalfagia@msu.edu}

\address[]{Department of Mathematics, Michigan State University, East Lansing, MI, 48824, USA}
\email[]{melbyjos@msu.edu}

\begin{document}

\begin{abstract}
We use Dehn surgery methods to construct infinite families of hyperbolic knots in the 3-sphere satisfying  a
weak form of the Turaev--Viro invariants volume conjecture. The results have applications to a
conjecture of Andersen, Masbaum, and Ueno about quantum representations of surface mapping class
groups. We obtain an explicit family of pseudo-Anosov mapping classes acting on surfaces of any genus and with one
boundary component that satisfy the conjecture.
\end{abstract}

\maketitle

\newtheorem{innercustomgeneric}{\customgenericname}
\providecommand{\customgenericname}{}
\newcommand{\newcustomtheorem}[2]{%
  \newenvironment{#1}[1]
  {%
   \renewcommand\customgenericname{#2}%
   \renewcommand\theinnercustomgeneric{##1}%
   \innercustomgeneric
  } {\endinnercustomgeneric} }

\newcustomtheorem{customthm}{Theorem}
\newcustomtheorem{customlemma}{Lemma}
\newcustomtheorem{customprop}{Proposition}
\newcustomtheorem{customconjecture}{Conjecture}
\newcustomtheorem{customcor}{Corollary}
\newcommand{\Q}{{\mathbb{Q}}} \newcommand{\R}{{\mathbb{R}}} \newcommand{\Z}{{\mathbb{Z}}}
\newcommand{\N}{{\mathbb{N}}} \newcommand{\C}{{\mathbb{C}}}

\theoremstyle{plain}
\newtheorem*{ack*}{Acknowledgements}
\newtheorem{thm}{Theorem}[section]
\newtheorem{lem}[thm]{Lemma}
\newtheorem{prop}[thm]{Proposition}
\newtheorem{defin}[thm]{Definition}
\newtheorem{cor}[thm]{Corollary}
\newtheorem{predefinition}[thm]{Definition}
\newtheorem{conjecture}[thm]{Conjecture}
\newtheorem{preremark}[thm]{Remark}
\newenvironment{remark}%
 {\begin{preremark}\upshape}{\end{preremark}} \newenvironment{definition}%
  {\begin{predefinition}\upshape}{\end{predefinition}}

\newtheorem{ex}[thm]{Example}
\newtheorem{ques}[thm]{Question}

\newcommand{\effie}[1]{{\color{blue}#1}}

\newcommand{\joe}[1]{{\color{red}#1}}

\section{Introduction}
The Turaev--Viro invariants of a compact $3$-manifold $M$ are a family of $\mathbb{R}$-valued
homeomorphism invariants $TV_r(M; q)$ parameterized by an integer $r\geq 3$ depending on a $2r$-th
root of unity $q$. They  were originally defined in terms of triangulations of compact $3$-manifolds
and were later related to skein-theoretic quantum invariants such as the Reshetikhin--Turaev and
colored Jones invariants \cite{TVCompact, RobertsSkein, growth6j, colJvolDKY}. In this paper, we are
primarily concerned with the geometric data recovered from the asymptotic behavior of the
Turaev--Viro invariants at the root $q=e^{\frac{2\pi i}{r}}$.

For a compact 3-manifold $M$, that is closed or has toroidal boundary, let
\begin{align*}
\textit{lTV}(M) := \liminf_{r\rightarrow \infty, \text{ } r \text{ odd}} \frac{2\pi}{r} &\log \left|\text{TV}_r \left(M; q = e^{\frac{2\pi i}{r}}\right)\right|,
\end{align*}

A 3-manifold $M$ with  $\textit{lTV}(M)>0$ is called  $q$\textit{-hyperbolic}. We will say that a
knot $K$ is $q$-hyperbolic if the complement $M_K:=\overline{S^3\setminus n(K)}$ is $q$-hyperbolic,
where  $n(K)$ is a tubular neighborhood of $K$. 

Chen and Yang \cite{chenyang2018vol} conjectured that if $M$ is hyperbolic, with volume ${\rm
vol}(M)$, then $\textit{lTV}(M)={\rm vol}(M)>0$. A related weaker conjecture, which was stated and
studied by Detcherry and Kalfagianni \cite{DKAdvances,detcherryKalfagianni2019gromov, DKIndiana}, is
the following:

\begin{conjecture} \label{conj: expgrowthconj} \textup{({Exponential Growth Conjecture})}
\textit{Let $M$ be a compact, oriented $3$-manifold with empty or toroidal boundary with Gromov norm
$||M||$. Then, $M$ is $q$-hyperbolic if and only if $||M||>0$.}
\end{conjecture}

By the geometrization theorem, a compact, oriented $3$-manifold M with empty or toroidal  boundary
can be cut along a canonical collection of tori into pieces that are either Seifert fibered
manifolds or hyperbolic. Moreover, $M$ has positive Gromov norm precisely when this decomposition
contains hyperbolic pieces. In this language, Conjecture \ref{conj: expgrowthconj} asserts that $M$
is $q$-hyperbolic if and only if its  geometric decomposition contains hyperbolic manifolds. One
direction of the conjecture, namely that if $M$ is  $q$-hyperbolic, then  $||M||>0$,  
follows from the main result of \cite{detcherryKalfagianni2019gromov}. The other direction was shown
in \cite{DKAdvances} to imply a conjecture of Andersen, Masbaum, and Ueno \cite{AMU} on the
geometric content of quantum representations of mapping class groups of surfaces.

The purpose of this paper is to give constructions of  the first  infinite families of hyperbolic
knots in the 3-sphere that are shown to be $q$-hyperbolic. The only knot complements in the 3-sphere
for which the asymptotic behavior of the Turaev--Viro invariants has been explicitly understood is
the figure-eight and all of its 2-cables \cite{colJvolDKY, 2-cable}. The only hyperbolic knot among
these is the figure-eight knot.  
On the other hand, the volume conjecture  of \cite{chenyang2018vol} has been proved for all
hyperbolic 3-manifolds that are obtained by Dehn filling the figure eight knot complement
\cite{OhtsukiFig8, wongYangF8}. Our constructions combine  these results, with a result of
\cite{detcherryKalfagianni2019gromov} about the behavior of the Turaev--Viro invariants under
Dehn-filling, and with  several Dehn surgery techniques. Our results also have new applications to
the conjecture of \cite{AMU}.

\subsection{Main results}
Given a knot $K$, let  $\mu, \lambda$ denote a set of canonical generators for
$H_1(\partial(n(K)))$. For a simple closed curve $s$ on $\partial(n(K))$, we denote by
$[s]=p\mu+q\lambda$ its class in $H_1(\partial(n(K)))$, where $p,q$ are relatively prime integers.
Recall that $s$ is completely determined, up to isotopy, by the fraction $p/q\in{\mathbb Q}\cup
\{\infty\}$. We will use $M_K({p}/{q})$ to denote the 3-manifold obtained by Dehn-filling $M_K$
along the slope $s$ determined by ${p}/{q}$.

For $m,n \in \mathbb{Z}$, let the \emph{double twist knot} $D(m,n)$ with $m$ vertical half-twists
  and $n$ horizontal half-twists be as in Figure \ref{fig:doubletwistknot}. For example, $D(2, -2)$
  is the figure-eight knot and $D(2, 2)$ is the left-handed trefoil.  With the exception of the
  unknot and the two trefoils,  the double twist knots are hyperbolic. We show the following:

  \begin{figure}
      \centering
      \includegraphics[scale=0.3]{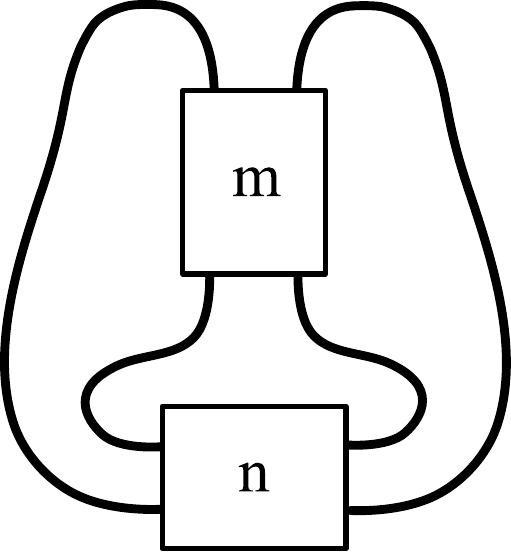}
      \caption{A double twist knot $D(m,n)$ diagram with $m$ vertical half-twists and $n$ horizontal
      half-twists.}
      \label{fig:doubletwistknot}
  \end{figure}

\begin{thm} \label{thm: doubleknotintro} For any integer $n\neq 0, -1$, the following are true:
\begin{enumerate}[(a)]
 \item The knots  $D_n:=D(2n, -3)$ and $D'_n:=D(2n, -2)$ are $q$-hyperbolic. \item The 3-manifolds
  $M_n:=M_{D_n}(4n+1)$ and $M'_n:=M_{D'_n}(1)$ are hyperbolic and $q$-hyperbolic. 

   \item We have
  $$\textit{lTV}(M_{D_n})\geq {\rm vol}(M_n), \ \ {\rm and} \ \  \textit{lTV}(M_{D'_n})\geq {\rm
  vol}(M'_n).$$
  
  \end{enumerate}
\end{thm}

Using Theorem \ref{thm: doubleknotintro}, we may conclude that many low-crossing knots are
$q$-hyperbolic. We refer to Tables \ref{tab: low crossing n} and \ref{tab: low crossing m} in
Section \ref{sec: low crossing tables}.

The knots $D(2n, -3)$ are fibered when $n <-1$ and the monodromies of their fibrations provide
explicit families of pseudo-Anosov  mapping classes acting on surfaces with a single boundary component that
satisfy the AMU conjecture. See Theorem \ref{AMUap}. These are the first examples known to satisfy
this conjecture that are constructed as monodromies of fibered knots in $S^3$. The examples of
\cite{DKAdvances} are coming from monodromies of fibered links of multiple components, while the
examples \cite{DKIndiana} come from monodromies of fibered knots in closed $q$-hyperbolic
3-manifolds.

A slope $p/q$ is called \emph{non-characterizing} for a knot $K\subset S^3$ if there is a knot $K'$
that is not equivalent to $K$ and such that $M_K(p/q)$ is homeomorphic to $M_{K'}(p/q)$. The
articles \cite{AJOT2013, AJLO2015, AbeTagami} give constructions of knots that admit
infinitely many non-characterizing slopes. Combining their techniques and results with Theorem
\ref{thm: doubleknotintro}, we are able to construct new infinite families of $q$-hyperbolic knots.

\begin{thm} \label{thm: 6_2annulustwist} There is an an infinite  set of knots ${\mathcal{K}}$ such
that:
  \begin{enumerate}[(a)]
  \item Every knot in ${\mathcal{K}}$ is $q$-hyperbolic.
   \item For every $K\in {\mathcal{K}}$, $M_{K}(-7) $ is homeomorphic to $M_{4_1}(-7/2)$ and it is
  $q$-hyperbolic.
  \item We have
  $$\textit{lTV}(M_{K})\geq {\rm vol}(M_{4_1}(-7/2))\approx 1.649610.$$
    \item No two knots in ${\mathcal{K}}$ are equivalent.
  \end{enumerate}
\end{thm}

To apply the methods of \cite {AJLO2015} one needs to start with a knot $K_0$ that admits an
``annulus presentation". Then, for any non-zero $n\in \N$, one applies a certain operation called an
``$n$-fold annulus twist" repeatedly to generate a family of knots ${\mathcal{K}}$, so that  for any
$K\in {\mathcal{K}}$ we have   $M_{K}(n)\cong M_{K_0}(n)$. The method of the proof of Theorem
\ref{thm: 6_2annulustwist}  is as follows: First we show that the six crossing knot $6_2$ is
$q$-hyperbolic and that the 3-manifold $M_{6_2}(-7)$, obtained by $-7$-surgery on $6_2$, is
homeomorphic to $M_{4_1}(-7/2)$ and  is $q$-hyperbolic. Then we verify that the knot $6_2$ has an
``annulus presentation" to which we apply a ``$-7$-fold annulus twist" inductively, to generate a
family of knots ${\mathcal{K}}$. In this case, the annulus presentation of $6_2$ is nice in a
certain sense (see 
Remark \ref{complicated}), and we are able to argue that the resulting knots have mutually distinct Alexander
polynomials. The reader is referred to Section  \ref{sec:further classes}  for the definitions of annulus presentations and annulus twists
and for the details of our construction.
The annulus twisting technique also applies to each of the
knots $D'_n:=D(2n, -2)$ to produce families of $q$-hyperbolic knots. However, in this case we don't
know whether the resulting knots are necessarily distinct. We have the following:

\begin{thm}\label{thm: doubleknot2seq} For any  $|n| > 1$, let $D'_n:= D(2n, -2)$. There is a
sequence knots $\{K_n^i\}_{i\in \mathbb{N}}$ such that, for any $i\in \N$, 
\begin{enumerate}[(a)]
\item the knot $K_n^i$ is $q$-hyperbolic;
\item the 3-manifold
$M_{K^i_n}(1)$ is homeomorphic to $M_{ D'_n}(1)$ and it is $q$-hyperbolic.

\end{enumerate}
\end{thm}

\subsection{Organization} The paper is organized as follows: We give a proof of Theorem \ref{thm:
doubleknotintro} in Section \ref{sec: q-hyp knots}. In Section \ref{sec:further classes}, first we
recall the definitions and results from \cite{AJOT2013, AJLO2015, AbeTagami} relevant here, and then
we prove Theorems \ref{thm: 6_2annulustwist} and \ref{thm: doubleknot2seq}. We apply our results to
the conjecture of \cite{AMU} in Section \ref{sec: AMUsection}. Finally,  in Section \ref{sec: low
crossing tables} we list all knots up  to ten crossings, and all knots from the SnapPy census of
hyperbolic cusped 3-manifolds with triangulation complexity at most nine,
that can be shown to be $q$-hyperbolic
using our methods.
\begin{ack*}
 {\rm The authors  thank Dave Futer for several helpful discussions during this project, and Futer, Purcell and Schleimer for
 generously sharing the data in Tables   \ref{tab: qhyp knots2-7} and \ref{tab: qhyp
  knots8-9}  with us.}
 \end{ack*}

\medskip

\section{\texorpdfstring{$q$}{q}-hyperbolic double twist knots}\label{sec: q-hyp knots} In this section, we will show
the $q$-hyperbolicity of two families of knots in $S^3$ which share hyperbolic Dehn surgeries with
the figure-eight knot.

Suppose $M$ is a compact 3-manifold with empty or toroidal boundary. If $M$ is hyperbolic, by Mostow
rigidity the volume of a hyperbolic metric is a topological invariant of $M$ denoted by ${\rm
vol}(M)$. If $M$ is disconnected the total volume is the sum of volumes over all connected
components. In general, by the geometrization theorem, $M$ admits a unique decomposition along tori
into manifolds with toroidal boundary that are Seifert fibered spaces or hyperbolic. Let $H_M$
denote the union of the hyperbolic components in the geometric decomposition of $M$. For the
purposes of this paper we define the Gromov norm of $M$ by
$$||M||:=\frac{ {\rm vol}(H_M)}{v_{\rm{tet}}},$$ where  $v_{\rm{tet}}= 1.01494\dots$ is the volume
of a regular ideal tetrahedron and ${\rm vol}(H_M)$ denotes the total volume of $H_M$. By work of
Thurston \cite{ThurstonGT3manifolds}, the Gromov norm of 3-manifolds with toroidal boundary
does not increase under Dehn-filling. That is,  if $M$ is a $3$-manifold with toroidal  boundary, and $M'$
is obtained by Dehn-filling of some components of $\partial M$, then $||M'||\leq ||M||$.

The asymptotics of the Turaev--Viro invariants have an analogous property, as shown  by
Detcherry-Kalfagianni \cite{detcherryKalfagianni2019gromov}.

\begin{thm}[\cite{detcherryKalfagianni2019gromov}, Corollary 5.3]\label{thm: lTV bounded} Let $M$ be
  a compact oriented 3-manifold with nonempty toroidal boundary and let $M'$ be a manifold obtained
  from $M$ by Dehn-filling some of the boundary components. Then 
  \[
  lTV(M') \leq lTV(M).  
  \]
  In particular, if $M'$ is $q$-hyperbolic then $M$ is $q$-hyperbolic.
\end{thm}

Let $K$ be a knot in the 3-sphere with complement $M_{K}$. Recall that isotopy classes of simple
closed curves on $\partial M_K$ are  in one to one correspondence with slopes $p/q\in \Q\cup
\{1/0\}$. Slopes of the form $p/1$ we will be denoted by $p$. Given a slope $p/q$, let $M_K(p/q)$
denote the 3-manifold obtained by $p/q$-surgery along $K$ (i.e.  $M_K(p/q)$ is obtained by a
Dehn-filling of $M_K$ along the simple closed curve of slope $p/q$ on $\partial M_K$). If $K$ is
hyperbolic and $M_K(p/q)$ is not hyperbolic, we say that $p/q$ is an \emph{exceptional slope} of
$K$.

Let $M_{4_1}$ denote  the complement of figure-eight knot $4_1$. The following is well known:

\begin{prop} \label{pro:ex} The set of the exceptional slopes of the knot figure-eight knot is
$E_{4_1}:= \{0, 1/0,  \pm 1, \pm 2, \pm 3, \pm 4 \}.$ Thus for any $p/q\notin E_{4_1}$ the
3-manifold $M_{4_1}(p/q)$ is hyperbolic.
\end{prop}

The asymptotics of the Turaev--Viro invariants of hyperbolic manifolds obtained by surgery on the
figure eight knot are well understood. Ohtsuki \cite{OhtsukiFig8} proved that hyperbolic manifolds
obtained by integral surgeries on $4_1$ satisfy the volume conjecture, and the result was extended
to rational surgeries by Wong and Yang \cite{wongYangF8}. 

\begin{thm}[\cite{OhtsukiFig8,wongYangF8}]\label{thm: Ohtsuki, Wong-Yang q-hyp} For any
non-exceptional slope $p/q$ of the knot $4_1$ we have
$$lTV(M_{4_1}(p/q))={\rm vol}(M_{4_1}(p/q)),$$ and hence, in particular, $M_{4_1}(p/q)$ is
$q$-hyperbolic.
\end{thm}

Recall that for  $m,n \in \mathbb{Z}$, we denote by $D(m,n)$ the double twist knot shown in Figure
  \ref{fig:doubletwistknot}.

 Next we construct two families of $q$-hyperbolic double twist knots parametrized by an integer $n$,
 which we denote by $D_n := D(2n, -3)$ and $D_n' := D(2n, -2)$.
  
\begin{customthm}{\ref{thm: doubleknotintro}} For any integer $n\neq 0, -1$, the following are true:
\begin{enumerate}[(a)]
 \item The knots  $D_n:=D(2n, -3)$ and $D'_n:=D(2n, -2)$ are $q$-hyperbolic. \item The 3-manifolds
  $M_n:=M_{D_n}(4n+1)$ and $M'_n:=M_{D'_n}(1)$ are hyperbolic and $q$-hyperbolic. 

   \item We have
  $$\textit{lTV}(M_{D_n})\geq {\rm vol}(M_n), \ \ {\rm and} \ \  \textit{lTV}(M_{D'_n})\geq {\rm
  vol}(M'_n).$$
  
  \end{enumerate}
\end{customthm}
    
 We will need the following lemma:

\begin{lem}\label{lem: twist knots from 4_1} For any $n\in \Z$ we have the following:
 \begin{enumerate}[(a)]
    \item The 3-manifold $M_{4_1}((-4n-1)/n)$ is homeomorphic to $M_{D_n}(4n+1)$.
    \item The 3-manifold $M_{4_1}(-1/n)$ is homeomorphic to $M_{D_n'}(1)$.
  \end{enumerate} 
\end{lem}

\begin{proof} For $n=0$ both $(a)$ and $(b)$ are trivially true: For, both $D_0, D'_0$ are the
trivial knot and we have:  $M_{4_1}(1/0)\cong M_{D_0}(1)=M_{D'_0}(1)\cong S^3$.

Next suppose that $n\neq 0$. Part $(a)$ follows from the fact that the 3-manifold
  $M_{4_1}(-(4n+1)/n)$  is related to $M_{D_n}(4n+1)$ by a sequence of Kirby-Rolfsen-Rourke
  moves. These moves are well known to preserve 3-manifolds up to homeomorphism. See for example
  \cite[Chapter 9]{Rolfsen}. The  particular sequence of moves required in this case  is shown  in Figure \ref{fig: 3,-n
  twist}. 
  
  To describe the moves required in more detail, let us recall that, as is customary in the
  Kirby-Rolfsen-Rourke calculus, one indicates the 3-manifold $M$ obtained by Dehn-filling along a
  link $L$ in $S^3$ by a diagram of $L$ with each component labeled by the surgery slope used for
  the component. For components where the surgery coefficient is $1/0$, we will omit the label (such
  surgery is called $\infty$-surgery and it produces back $S^3$). Such a diagram is called a
  {\emph{surgery diagram}} of $M$.
  
  \begin{figure}
    \centering
    \includegraphics*[scale=0.69]{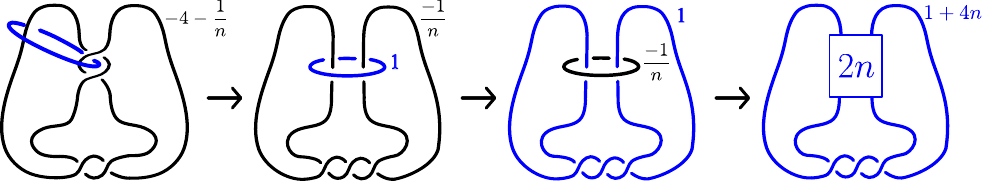}
    \caption{Kirby-Rolfsen-Rourke calculus moves showing that $M_{4_1}((-4n-1)/n)$ is homeomorphic
    to $M_{D_n}(4n+1)$.} 
    \label{fig: 3,-n twist}
  \end{figure}
  
    \begin{enumerate} [(i)]
    
    \item The 3-manifold $M_{4_1}(-(4n+1)/n)$ has  a surgery diagram consisting of a knot diagram
    for  $4_1$ labeled by $(-4n-1)/n=-4-1/n$. In the leftmost panel of  Figure \ref{fig: 3,-n
    twist}, we have inserted an unknotted component $U$, shown in blue, on which the surgery
    coefficient is $1/0=\infty$ and such that it has linking number $\pm 2$ with the figure-eight
    knot component. That is $| {\rm lk}(U, 4_1)|=2$. A $-1$-twist  along $U$ produces produces the
    second surgery diagram in the sequence. Note that the surgery coefficient of the component
    corresponding to $4_1$ has now changed to $-(4n+1)/n+({\rm lk}(U, 4_1))^2=-1/n$. This operation
    is also known as a \emph{blow up}.

    \item The surgery diagram shown in the  third panel of Figure \ref{fig: 3,-n twist} is obtained
    by that of the second panel by ambient isotopy that interchanges the two components of the
    underlying link.
    
    \item Finally, performing $n$-twists on the component labelled by $-1/n$  gives the rightmost
    panel of Figure \ref{fig: 3,-n twist}, which represents a surgery diagram of $M_{D_n}(4n+1)$.
    The operation of performing this $(-1/n)$-surgery on an unknotted component is also known as a
    \emph{blow down}.
  
 \end{enumerate} 

 For part $(b)$, a similar sequence of  Kirby-Rolfsen-Rourke calculus moves, shown in Figure
  \ref{fig: 2,-n twist}, in proves that  $M_{4_1}(-1/n)$ is homeomorphic $M_{D_n'}(1)$. Note that
  this time the inserted unknotted component $U$, drawn in blue in the leftmost panel of Figure
  \ref{fig: 2,-n twist}, has zero linking number with $4_1$. That is, $ {\rm lk}(U, 4_1)=0$. In this
  case, the surgery coefficient of the component corresponding to $4_1$ is unchanged under the blow
  up operation since $-1/n+({\rm lk}(U, 4_1))^2=-1/n$.
   \begin{figure}
    \centering
    \includegraphics*[scale=0.85]{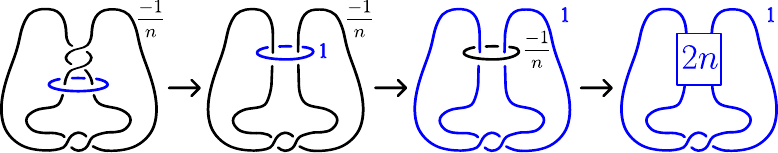}
    \caption{Kirby-Rolfsen-Rourke calculus moves showing that $M_{4_1}(-1/n)$ is homeomorphic to
    $M_{D_n'}(1)$.} 
    \label{fig: 2,-n twist}
  \end{figure}
\end{proof}

We are now ready to give the proof of Theorem \ref{thm: doubleknotintro}:

\begin{proof}[Proof of Theorem \ref{thm: doubleknotintro}]
  By Lemma \ref{lem: twist knots from 4_1}, for any $n\in \Z$, the 3-manifolds $M_n:=M_{D_n}(4n+1)$
  is obtained by $-(4n+1)/n$-surgery along the knot $4_1$. Since $n\neq 0,-1$, by Proposition
  \ref{pro:ex}, the slope $-(4n+1)/n$ is not exceptional for $4_1$. Hence $M_n$ is hyperbolic.
  Similarly, since $M'_n:=M_{D_n'}(1)$ is also obtained by a $-1/n$-surgery along $4_1$, it is
  hyperbolic for $n\neq 0,\pm 1$. By Theorem \ref{thm: Ohtsuki, Wong-Yang q-hyp}, we conclude that
  the manifolds $M_{D_n}(4n+1)$ and $M_{D_n'}(1)$ are $q$-hyperbolic for $n \neq \pm 1$. Hence, part
  $(b)$ of the statement follows.

   Theorem \ref{thm: lTV bounded} implies that the growth rates of the Turaev--Viro invariants of
  the unfilled twist knot complements $M_{D_n}$ and $M_{D_n'}$ are bounded below by the growth rates
  of $M_n$ and $M_n'$ respectively. That is we have
  $$0< lTV(M_n)\leq lTV(M_{D_n})\ \ {\rm and} \ \  0< lTV(M_n')\leq lTV(M_{D_n'}).$$ Hence, by their
  definitions, the double twist knots $D_n$ and $D_n'$ are also $q$-hyperbolic, concluding the proof
  of part $(a)$.
  
  Now we prove part $(c)$: Since $M_n$ and $M_n'$ are hyperbolic 3-manifolds obtained by surgery of
  $4_1$, by Theorem \ref{thm: Ohtsuki, Wong-Yang q-hyp}, $lTV(M_n)={\rm vol}(M_n)$ and
  $lTV(M'_n)={\rm vol}(M'_n)$. Combining these equations with the last displayed inequalities gives
  part $(c)$.

\end{proof}

\medskip


\section{Non-characterizing slopes and \texorpdfstring{$q$}{q}-hyperbolicity}\label{sec:further
classes} A slope $p/q$ is called \emph{non-characterizing} for a knot $K\subset S^3$ if there is a
knot $K'$ that is not equivalent to $K$ and such that $M_K(p/q)$ is homeomorphic to $M_{K'}(p/q)$.
For the viewpoint of this paper, non-characterizing slopes are useful in the following sense: If we
know that $M_K(p/q)$ is $q$-hyperbolic then, arguing as in the proof Theorem \ref{thm:
doubleknotintro}, we conclude that 
$$0<lTV(M_{K}(p/q))=lTV(M_{K'}(p/q))\leq lTV(M_{K'}),$$ and hence $K'$ is a $q$-hyperbolic knot.

In the articles, \cite{AJOT2013, AJLO2015, AbeTagami}, the authors provide constructions of knots
that admit many non-characterizing slopes.  The techniques of these papers apply to many double
twist knots to conclude that they admit non-characterizing slopes. On the other hand, these knots
can be seen to be $q$-hyperbolic  by Theorem \ref{thm: doubleknotintro}. Using this approach, one
starts with a double twist knot, say $K,$ to which both the techniques of \cite{AJOT2013, AJLO2015,
AbeTagami} and Theorem \ref{thm: doubleknotintro}  apply, and builds a  family of $q$-hyperbolic
knots that have a common surgery with $K$.

To illustrate this, we note that the knot $6_2$ is isotopic to the double twist knot $D(-4, -3)$; we
will write $6_2=D(-4, -3)$. See Section \ref{sec: low crossing tables} for more details. By Theorem
\ref{thm: doubleknotintro}, $M_{6_2}(-7) \cong M_{4_1}(-7/2)$ and $6_2$ is $q$-hyperbolic. We will
use the approach discussed above to prove the following theorem stated in the Introduction:

\begin{customthm}{\ref{thm: 6_2annulustwist}} There is an infinite  set of knots ${\mathcal{K}}$
such that:
  \begin{enumerate}[(a)]
  \item Every knot in ${\mathcal{K}}$ is $q$-hyperbolic.
   \item For every $K\in {\mathcal{K}}$, $M_{K}(-7) $ is homeomorphic to $M_{4_1}(-7/2)$ and it is
  $q$-hyperbolic.
  \item We have
  $$\textit{lTV}(M_{K})\geq {\rm vol}(M_{4_1}(-7/2))\approx 1.649610.$$
    \item No two knots in ${\mathcal{K}}$ are equivalent.
  \end{enumerate}
\end{customthm}

In order to prove Theorem \ref{thm: 6_2annulustwist} and to discuss further applications of the
techniques of \cite{AJOT2013, AJLO2015, AbeTagami} in constructions of $q$-hyperbolic knots, we need
some preparation.

\subsection{Annulus presentations and twists}
We begin by recalling the notion of \emph{annulus presentations} of knots and the operation of
\emph{annulus twists} for knots admitting annulus presentations. The latter operation takes a
surgery presentation along a particular class of knots and returns a different knot which shares a
surgery with the original knot.

\begin{defin}\label{defin:annulus} \rm{We will say that a knot $K \subset S^3$ admits an annulus
presentation if it can be constructed in the following way:
\begin{enumerate}
\item Start with standardly embedded annulus $A \subset \mathbb{R}^2 \cup \{\infty\} \subset S^3$
together with an an unknotted curve $c$ that is disjoint from $A$ that bounds a disc $\Sigma$ whose
interior intersects $\partial A$ twice; once for each component of $\partial A$. Consider $c$ as a
framed knot with framing $\pm 1$. 

\item Consider an embedded band $b:I\times I \rightarrow S^3$ such that 
\begin{enumerate}[(i)]
  \item $b(I\times I) \cap \partial A = b(\partial I \times I)$,
  \item $b(I \times I) \cap$int$A$ consists of ribbon singularities, 
  \item $A\cup b(I\times I)$ is an immersed orientable surface, and
  \item $b(I\times I) \cap c = \emptyset$,
\end{enumerate}
where $I = [0,1]$. See the right hand side panel of Figure \ref{fig: annulus pres} for an
  illustration of an annulus presentation $(A,b,c)$.  
\item Performing the $\pm 1$ surgery on $c$ (i.e blowing down along $c$) transforms the curve
$(\partial A \setminus b(\partial I \times I)) \cup b(I \times \partial I)$ into a knot that is
isotopic to $K$ in $S^3$. \end{enumerate}}
\end{defin}

\begin{figure}
  \centering
  \includegraphics*[scale=0.55]{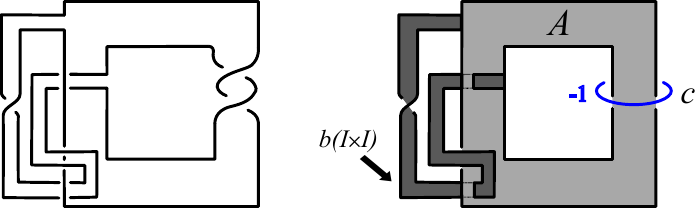}
  \caption{Annulus presentation of the knot $6_2$ in $S^3$.}
    \label{fig: annulus pres}
\end{figure}

\begin{remark}
  We note that the definition of annulus presentation differs slightly across the literature.
  Namely, in \cite{AbeTagami}, the authors use a more general definition of annulus presentation that
  allows the annulus $A$ to be any embedding. They define a \textit{special} annulus presentation
  equivalently to the above definition except that the presentation includes the single full
  crossing (either positive or negative) in the Hopf band resulting from surgery along the $(\pm
  1)$-framed unknotted component $c$. Here we use the definition given by Abe--Jong--Omae--Takeuchi
  \cite{AJOT2013} and Abe--Jong--Luecke--Osoinach \cite{AJLO2015}. Note that in \cite{AJOT2013}, the
  authors use the term \textit{band presentation} rather than annulus presentation.
\end{remark}

To continue, note that given an annulus presentation  $(A,b,c)$, the complement of the annulus $A
\subset \mathbb{R}^2 \cup \{\infty\}$ consists of two disk components $D$ and $D'$. Take $D$ to be the
component corresponding to the finite region in ${\R}^2$ (see the leftmost panel of Figure \ref{fig:
6_2 5_2 annulus}) and assume that $\infty \in D'$.

\begin{defin} \label{defin:simple} \rm{ The annulus presentation $(A,b,c)$ is called \textit{simple}
if we have $b(I \times I) \cap \text{int}D = \emptyset$. }
\end{defin}

The middle panel of Figure \ref{fig: 6_2 5_2 annulus} illustrates  a simple annulus presentation of
the knot $6_2$  while the rightmost panel illustrates a non-simple annulus presentation for the knot
$5_2$.

\begin{figure}
  \centering
  \includegraphics[scale=0.54]{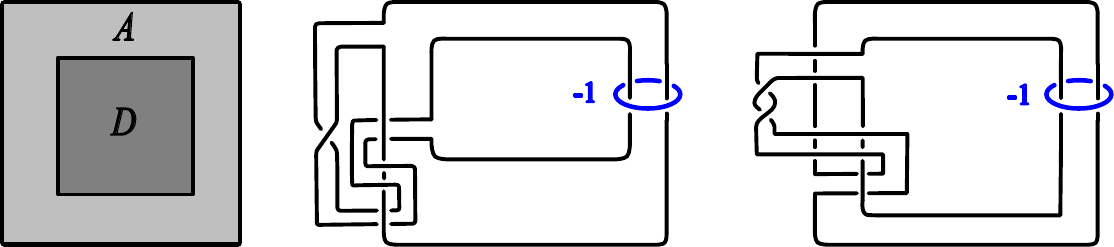}
  \caption{\textbf{Left:} One of the two connected components, $D$, of $\mathbb{R}^2 \cup \{\infty\}
  \setminus \text{int}A$. \textbf{Middle:} Simple annulus presentation of the knot $6_2$.
  \textbf{Right:} Non-simple annulus presentation of the knot $5_2$.}
  \label{fig: 6_2 5_2 annulus}
\end{figure}

\begin{figure}
  \centering
  \includegraphics[scale=0.50]{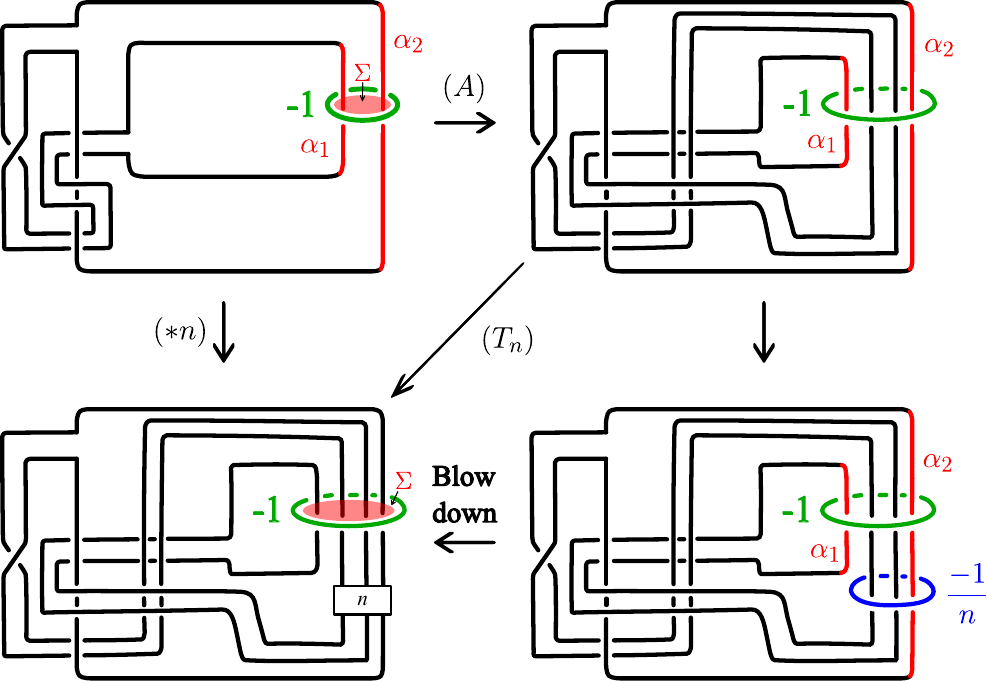}
  \caption{ \textbf{Top row:} Simple annulus presentation of $6_2$ and the annulus twist (A).
   \textbf{Bottom row:} Introduction of $(-1/n)$-framed component and blow down.} 
   \label{fig: ann twist sequence}
\end{figure}

The following lemma of \cite{AJOT2013} gives a family of knots which admit annulus presentations. In
particular, the double twist knots $D_n' = D(2n, -2)$, including those listed in Table \ref{tab: low
crossing m}, satisfy the assumptions of Lemma \ref{lem: unknotting 1}. 

\begin{lem}[\cite{AJOT2013}, Lemma 2.2]\label{lem: unknotting 1} If $K$ is a knot with unknotting
  number one, then $K$ admits an annulus presentation.
\end{lem}

Abe and Tagami \cite{AbeTagami} give a tabulation of all prime knots with $8$ or fewer crossings
admitting an annulus presentation (see Table 1 of \cite{AbeTagami}).

We now define an operation known as an \emph{$n$-fold annulus twist}. We refer the reader to
\cite{AJLO2015,AJOT2013} for further details of this construction. This operation can be  applied to
a knot $K$ with an annulus presentation $(A,b,c)$, and surgery slope given by an integer $n\in \Z$,
to produce another knot $K'$, with annulus presentation $(A,b',c)$, so that the 3-manifold
$M_K(n)\cong M_{K'}(n)$.

\begin{definition}\label{def: n-fold annulus twist} Let $K$ be a knot with annulus presentation
  $(A,b,c)$ with $\partial A = l_1 \sqcup l_2$, and let $n\in \mathbb{Z}$. We define the
  \emph{$n$-fold annulus twist operation}, denoted by $(*n)$, as follows:
  \begin{enumerate}
    \item First apply an annulus twist $(A)$. This involves performing Dehn surgery on $l_1$ and
    $l_2$ along slopes $1$ and $-1$, respectively, and gives rise to a homeomorphism of the
    complement $M_{l_1 \sqcup l_2}$. An example is illustrated in the top row of Figure \ref{fig:
    ann twist sequence}. Note that in the leftmost panel we have two vertical arcs $\alpha_1,
    \alpha_2 \subset \partial A$ that intersect the interior of a disk $\Sigma$ bounded by the $-1$
    framed unknot $c$ exactly twice. After the operation (A) is applied, the disk $\Sigma$ is
    intersected by four vertical arcs, two of which are between $\alpha_1$ and $\alpha_2$.
   \item Apply the operation $(T_n)$, which is defined by 
    \begin{enumerate}[(i)]
   \item adding another  $(-1/n)$-framed unknot  engulfing all but $\alpha_1$ of the
    vertical arcs going through $c$. An illustration is given in the rightmost panel of
     the second row of Figure \ref{fig: ann twist sequence}.
 \item blowing down along the $(-1/n)$-framed component, as shown in the leftmost panel
   of the second row of Figure \ref{fig: ann twist sequence}.
   \end{enumerate}
  \end{enumerate}
\end{definition}

An important property of the $n$-fold annulus twist operation is the following result of
Abe-Jong-Luecke-Osoinach \cite{AJLO2015}. 

\begin{thm}[\cite{AJLO2015}, Theorem 3.10]\label{thm: ann pres same mfd} Let $K$ be a knot with an
annulus presentation and $K'$ be the knot obtained by the $n$-fold twist $(*n)$. Then the 3-manifold
$M_K(n)$ is homeomorphic to $M_{K'}(n)$. That is we have
  \[M_K(n) \cong M_{K'}(n). \]
\end{thm}

A proof of Theorem \ref{thm: ann pres same mfd} for the knot $K=6_2$ is given in Figure \ref{fig:
homeoanntwist}, which is summarized as follows:
\begin{enumerate} [(i)]

\item First we perform a blow up operation, which changes the framing of the $n$-framed component to $0$ and introduces
a $(-1/n)$-framed component as shown in the middle panel of the first row.

\item After introducing $1$ and $-1$-framed components (in red) in the right most panel of the first
row, we slide the $1$-framed component across the $-1$-framed component to get the right most panel
of the second row. Note that the $1$ and $-1$-framed components (in red) in the right most panel of
the second row correspond to the boundary components of the annulus $A$ and give the surgery
description for the move $(A)$.

 \item Next we slide the $(-1/n)$-framed component across both the $1$-framed component (in red) and
the $-1$-framed component (in green) to get the left most panel of the third row. 

\item To get from the left most panel to the middle panel of the third row, we perform surgery on
the red $1$ and $-1$-framed components, corresponding to the annulus twist $(A)$, and isotope the
$(-1/n)$-framed component. Finally, we blow down, which introduces $n$ full positive twists and
changes the framing from $0$ to $n$. This isotopy and blow down correspond to the operation $(T_n)$.
Hence the sequence of operations in the third row contains an $(A)$ move and a $(T_n)$ move.
\end{enumerate}
\begin{figure}
  \centering
  \includegraphics[scale=0.55]{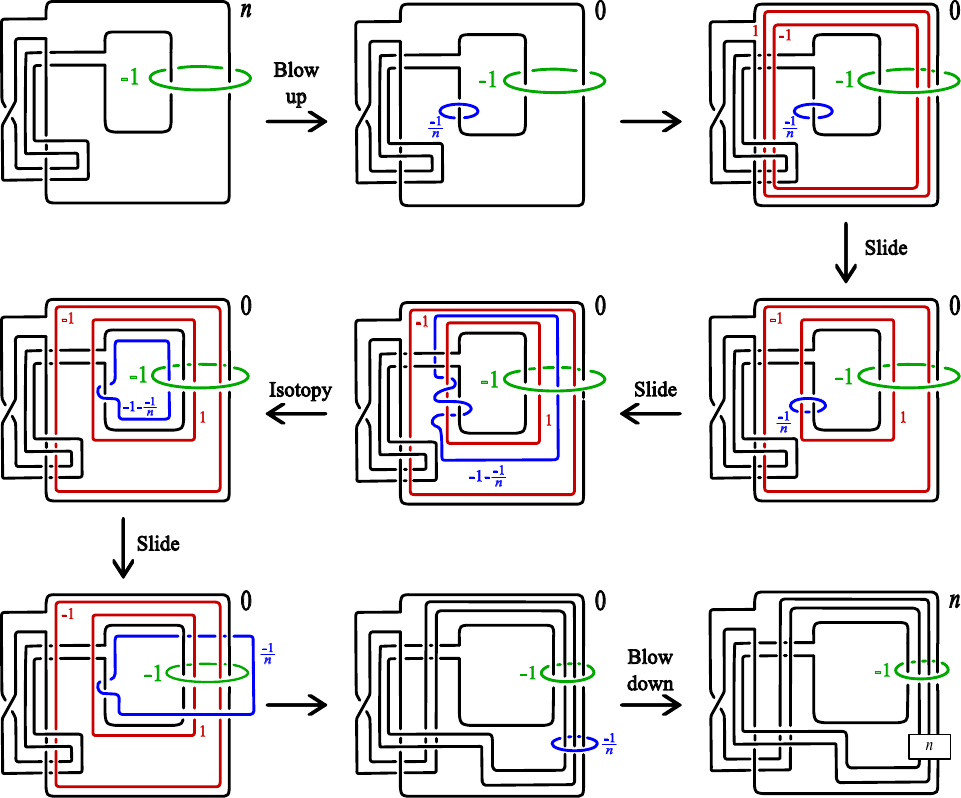}
  \caption{A proof that $M_{K}(n) \cong M_{K'}(n)$ for $K=6_2$ starting with an annulus presentation
  of $K$ in the top-left and ending with an annulus presentation of $K' = (*n)K$ in the
  bottom-right.}
  \label{fig: homeoanntwist}
\end{figure}

\begin{remark}\label{complicated} If a knot $K$ admits an annulus presentation and a knot $K'$ is
    obtained from $K$ by an $n$-fold annulus twist $(*n)$, then, in general, $K'$ can be far more
    complicated than $K$. However, if $K$ admits a simple annulus presentation, then the annulus
    presentation of $K'$ is also simple and is not quite as complicated.
\end{remark}

Since the $n$-fold annulus twist operation on a knot produces another knot which also admits an
annulus presentation, this operation can be iterated. Indeed, Theorem \ref{thm: ann pres same mfd}
implies that for any knot $K$ which admits an annulus presentation and any integer $n\neq 0$, there
is a set ${\mathcal{K}}=\{K_i\}_{i\in \N}$ of knots such that

$$ ....M_{K_i}(n) \cong M_{K_{i-1}}(n)\cong\cdots \cong M_{K_1}(n) \cong M_{K}(n).$$

In general, we don't know that  the knots $K_i$ are necessarily distinct, so  the set $\mathcal{K}$
may be finite. However, we will see in the proof of Theorem \ref{thm: 6_2annulustwist} that in the
case of $6_2$, iterating the twist operation produces an infinite sequence of mutually distinct
knots.

\subsection{Applications to \texorpdfstring{$q$}{q}-hyperbolicity} 
In order to prove Theorem \ref{thm: 6_2annulustwist}, we recall some definitions from Section 3.3.1
of \cite{AJLO2015}. There the authors use the surgery description of the infinite cyclic covering
$\Tilde{E}(K)$ of the exterior $E(K)$ of a knot $K$ to distinguish knots obtained by applying the
operation $(*n)$ iteratively, provided that the annulus presentation of the knot to begin with is
``good'' in the sense of Definition \ref{def: goodAP} below.

Let $K$ be a knot with a simple annulus presentation $(A,b,c)$. If we ignore the $(-1)$-framed loop
$c$, the knot $U:=(\partial A \setminus b(\partial I \times I)) \cup b(I \times \partial I)$ is
trivial in $S^3$. Consider the link $U\cup c$ in $S^3$. The component $U$ bounds an immersed disk with ribbon singularities
while the component $c$ bounds an embedded disk
$\Sigma$ whose interior is pierced twice by $U$.
We may isotope $U\cup c$ so that the immersed disk bounded by $U$ becomes an embedded
 flat disk, denoted by $D$, contained
in  $\mathbb{R}^2\subset (\mathbb{R}^2 \cup \{\infty\})$. This isotopy, which we denote by $\phi$, will gradually
 shrink the band $b(I\times I)$ till it is eliminated, and will introduce ribbon singularities between the  (isotopic image of $c$)
 and the interior of $D$.  Abusing our notation, we will continue to denote the image of $c$ under $\phi$ by $c$, and
 we will continue to use $\Sigma$ to denote the image  
 of $\Sigma$ under $\phi$. Note that after  the isotopy $\Sigma$ may become an immersed disk.
 We also note that in Figure \ref{fig: isotopygoodAP6_2} (middle panel)
 the disk $\Sigma$ after the isotopy is not entirely depicted shaded; we only indicate the shading in a small portion
  we wish to highlight in Figure \ref{fig: isotopygoodAP6_2_twisted} below where we show the ribbon singularities of a disk
  $\Sigma$ after isotopy of  $U\cup c$  that makes $U$ bound a flat disk $D$.
 
 Fix orientations on $U$ and $c$ and cut the complement of $U$ after the isotopy in
$S^3$ along the flat disk $D$. This gives a solid cylinder $D\times [-1, 1]$. We will denote the two copies of $D$
resulting from this cutting by $D_{-1}$ and $D_1$. The cutting separates the oriented loop $c$ (after the isotopy) into
a set $\mathcal{A}$ oriented arcs with endpoints on $D_{\pm 1}$, and the endpoints of each  arc
$\alpha \in \mathcal{A}$ may be labelled by $``+"$ (resp. ``$-$") according to whether the algebraic
intersection number of $\alpha$ with the disk $D$ ilies on is positive (resp. negative). This
categorizes $\alpha$ as  one of four types: $(++), (--), (+-), \rm{ and } (-+)$. An illustration of
the process for the $6_2$ knot is shown in Figure \ref{fig: isotopygoodAP6_2}. We refer the reader
to \cite[Section 3.3.1]{AJLO2015} for further details of this construction. 
\smallskip

We need the following definition of \cite{AJLO2015}.
\begin{figure}
  \includegraphics[scale=0.6]{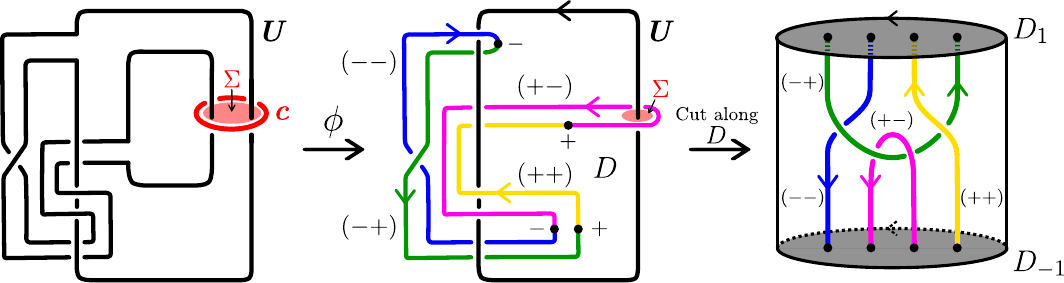}
  \caption{The isotopy $\phi$ applied to the simple annulus presentation of $6_2$. In the collar $D
           \times [-1,1]$, the $(+-)$ arc (in pink) and the $(-+)$ arc (in green) have linking
           number $-1$ relative to $D_{-1}\sqcup D_1$.}
  \label{fig: isotopygoodAP6_2}
\end{figure}

\begin{definition}(\cite[Definition 3.14]{AJLO2015})\label{def: goodAP} A simple annulus
  presentation $(A,b,c)$ is \textit{good} if $b(I \times \partial I) \cap int A \neq \emptyset$ and
  the set of arcs $\mathcal{A}$ in $D \times [-1,1]$ obtained by cutting along $D$ satisfies the
  following up to isotopy.
  \begin{enumerate}[(1)]
    \item $\mathcal{A}$ contains exactly one $(+-)$ arc and exactly one $(-+)$ arc, and the linking
    number of these arcs $\rm{rel}(D_{-1} \sqcup D_1)$ is $\pm 1$. 
    \item For $\alpha \in \mathcal{A}$, if $\alpha \cap \rm{int}\Sigma \neq \emptyset$, then
    $\alpha$ is of type $(++)$ (resp. $(--)$) and the sign of each intersection point in $\alpha
    \cap \Sigma$ is $+$ (resp. $-$). Note that here we refer to the immersed  $\Sigma$ after the isotopy $\phi$,
    and so the arcs $\mathcal{A}$ intersect $\rm{int}\Sigma$.

  \end{enumerate}
\end{definition}

\begin{figure}
  \includegraphics[scale=0.6]{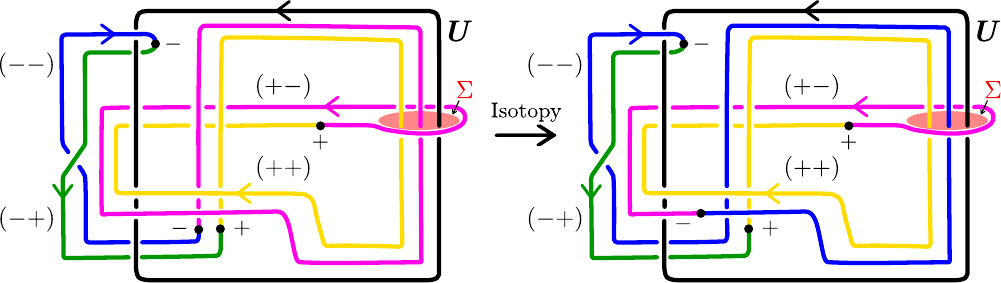}
  \caption{\textbf{Left:} The isotopy $\phi$ applied to the simple annulus presentation of $(A)
  6_2$. In this form, Part $(2)$ of Definition \ref{def: goodAP} fails. \textbf{Right:} Good annulus
  presentation of $(A)6_2$.}
  \label{fig: isotopygoodAP6_2_twisted}
\end{figure}
To illustrate and motivate Definition \ref{def: goodAP}, we apply the isotopy $\phi$ shown in Figure
\ref{fig: isotopygoodAP6_2} for the simple annulus presentation of $6_2$ to the simple annulus
presentation obtained by applying the annulus twist $(A)$ shown in the top row of Figure \ref{fig:
ann twist sequence}. The left panel of Figure \ref{fig: isotopygoodAP6_2_twisted} shows the result
of applying $\phi$ to $(A)6_2$. However, since there is a $(+-)$ arc intersecting the interior of
$\Sigma$, it does not satisfy Part $(2)$ of Definition \ref{def: goodAP}.

To remedy this, we apply an isotopy to move the $(-)$ intersection point between the disk $D$ and
the $(+-)$ and $(--)$ arcs, as shown in the right panel of Figure \ref{fig:
isotopygoodAP6_2_twisted}. This isotopy moves this intersection through $\Sigma$ and results in a
diagram satisfying Part $(2)$ of Definition \ref{def: goodAP}. In particular, since the only arcs
intersecting $\rm{int}\Sigma$ are of type $(++)$ and $(--)$, the diagram in the right panel of
Figure \ref{fig: isotopygoodAP6_2_twisted} corresponds to a good annulus presentation. After this
isotopy, further applications of the annulus twist $(A)$ leave the $(+-)$ and $(-+)$ arcs fixed
since they are now disjoint from $\rm{int}\Sigma$.

The importance of having a good annulus presentation for a knot $K$ lies in the fact that, as shown
in \cite{AJLO2015}, the number of intersection points between $(++)$ arcs and $\rm{int}\Sigma$
determines the degree of its Alexander polynomial. As Figure \ref{fig: isotopygoodAP6_2_twisted}
illustrates, each annulus twist increases the number of such intersections with $\rm{int}\Sigma$,
hence increasing the degree of the Alexander polynomial.

The following lemma of \cite{AJLO2015} will be used in the proof of Theorem \ref{thm:
6_2annulustwist}. 

\begin{lem}[\cite{AJLO2015}, Lemma 3.12]\label{lem: goodAPdifferentAlex} Let $n\in \mathbb{Z}$ and
  suppose the knot $K$ admits a good annulus presentation. Let $K'$ be the knot obtained by applying
  the operation $(*n)$ to $K$. Then, we have the following:
  \begin{enumerate}[(a)]
    \item The knot $K'$ also admits a good annulus presentation.
    \item If $\Delta_K(t)$ and  $\Delta_{K'}(t)$ denote the  Alexander polynomials of $K$ and $K'$,
    respectively, then $\rm{deg} \Delta_K(t) < \rm{deg} \Delta_{K'}(t)$.
  \end{enumerate}
\end{lem}

\begin{remark}\label{rem: goodAPMonic} As shown in \cite{AJLO2015},  if a knot $K$ admits a good
 annulus presentation, then its Alexander polynomial $\Delta_K(t)$ is monic. 
\end{remark}

We may now prove Theorem \ref{thm: 6_2annulustwist}.

\begin{proof}[Proof of Theorem \ref{thm: 6_2annulustwist}] 
  As noted earlier the knot $K_0:=6_2$ admits a simple annulus presentation. We will consider $K_0$
with framing $-7$ and apply the sequence of moves in Figure \ref{fig: homeoanntwist} to obtain a
knot $K_1$  with simple annulus presentation and such that $M_{K_1}(-7) \cong M_{6_2}(-7)\cong
M_{4_1}(-7/2)$. See Theorem \ref{thm: ann pres same mfd}. Since as discussed earlier $M_{6_2}(-7)$
is $q$-hyperbolic, we obtain that $K_1$ is $q$-hyperbolic and $M_{K_1}(-7) $ is $q$-hyperbolic. Now
we can repeat the process for the knot $K_1$, and apply Theorem \ref{thm: ann pres same mfd} again
to obtain a knot $K_2$ with simple annulus presentation and such that $M_{K_2}(-7) \cong
M_{K_1}(-7)$. Inductively, we create a set  ${\mathcal{K}}=\{K_i\}_{i\in \N}$ of knots with simple
annulus presentations such that
$$ ....M_{K_i}(-7) \cong M_{K_{i-1}}(-7)\cong\cdots \cong M_{K_1}(-7)\cong M_{6_2}(-7) \cong
M_{4_1}(-7/2).$$ By construction, each $K_i$ and $M_{K_i}(-7)$ are $q$-hyperbolic, hence the
collection ${\mathcal{K}}$ satisfies parts $(a)$-$(b)$ of the statement of the theorem.

By Theorem \ref{thm: Ohtsuki, Wong-Yang q-hyp}, $1.649610\approx {\rm
  vol}(M_{4_1}(-7/2))=\textit{lTV}(M_{4_1}(-7/2))$. Combining this with part $(b)$ and Theorem
  \ref{thm: lTV bounded}, we get

  $$\textit{lTV}(M_{K_i})\geq \textit{lTV}(M_{4_1}(-7/2))= {\rm vol}(M_{4_1}(-7/2))\approx
  1.649610,$$ obtaining part $(c)$ of the theorem statement.

Next we claim that the knots $K_i\in {\mathcal{K}}$ are distinct. Let $U$ be the trivial knot in
$S^3$ obtained by ignoring the $(-1)$-framed loop $c$ in the simple annulus presentation of $K_0$
(see the middle panel of Figure \ref{fig: 6_2 5_2 annulus}). Let $D$ be the disk bounded by $U$, and
let $\Sigma$ be the disk bounded by $c$. Figure \ref{fig: isotopygoodAP6_2} gives the isotopy that
flattens $D$ so that it is contained in $\mathbb{R}^2 \subset (\mathbb{R}^2 \cup \{\infty\})$. Let
$D_{-1}=D \times \{-1\}$ and  $D_1=D^2 \times \{1\}$ be copies of $D$ in the bundle $D^2 \times
[-1,1]$ obtained by cutting along $D$. Note that the resulting set of oriented arcs $\mathcal{A}$
shown in the right panel of Figure \ref{fig: isotopygoodAP6_2} contains exactly one $(+-)$ arc and
exactly one $(-+)$ arc. Relative to $D_{-1} \sqcup D_1 \subset D \times [-1,1]$, the $(+-)$ arc (in
pink) links with the $(-+)$ arc (in green) with linking number $-1$. Moreover, $\alpha \cap \Sigma =
\emptyset$ for any arc $\alpha \in \mathcal{A}$. By Definition \ref{def: goodAP}, $K_0$ admits a
good annulus presentation.

By Lemma \ref{lem: goodAPdifferentAlex}, every $K_i \in \mathcal{K}$ admits a good annulus
presentation, and the Alexander polynomials of this family satisfy
$$\rm{deg}\Delta_{K_0}(t) < \rm{deg}\Delta_{K_1}(t) < \cdots <\rm{deg}\Delta_{K_{i-1}}(t)
<\rm{deg}\Delta_{K_i}(t) < \cdots$$ This establishes part $(d)$ of the statement of the theorem.
\end{proof}

The above argument applies to any $q$-hyperbolic knot which admits a good annulus presentation and
an integer $q$-hyperbolic Dehn-filling. For any such knot, analogously to $6_2$, one may apply the
same procedure to produce an infinite family of distinct $q$-hyperbolic knots with homeomorphic
$n$-surgeries. For instance, the method can be applied to the knot $8_{20}$. See section \ref{sec:
low crossing tables} for more details. Hence we have the following:
 
 \begin{thm}Suppose that $K$ is knot that admits a good annulus presentation and such that $M_K(n)$
 is $q$-hyperbolic for some $0\neq n \in Z$. Then there is an infinite family $\{ K_i\}_{i\in \N}$
 of distinct $q$-hyperbolic knots, such that $M_{K_i}(n) \cong M_{K}(n)$, for any ${i\in \N}$.
 
 \end{thm}

We now turn our attention to the $q$-hyperbolic knots $D_n' = D(2n, -2)$ and their $q$-hyperbolic
fillings $D_n'(1)$. It is known that these double twist knots have unknotting number $1$, hence
admit an annulus presentation by Lemma \ref{lem: unknotting 1}. This gives rise to the following
theorem.

\begin{customthm}{\ref{thm: doubleknot2seq}}For any  $|n| > 1$, let $D'_n:= D(2n, -2)$. There is a
sequence of knots $\{K_n^i\}_{i\in \mathbb{N}}$ such that for any $i\in \N$,
the following:
\begin{enumerate}[(a)]
\item the knot $K_n^i$ is $q$-hyperbolic;
\item the 3-manifold $M_{K^i_n}(1)$ is homeomorphic to $M_{D'_n}(1)$ and it is $q$-hyperbolic.
\end{enumerate}
\end{customthm}

\begin{proof}
Fix $|n| > 1$. The  double twist knot  $D'_n:= D(2n, -2)$ has unknotting number 1 and hence by Lemma
\ref{lem: unknotting 1}, it admits an annulus presentation. Let $K_n^0:= D'_n$. By Theorem \ref{thm:
ann pres same mfd} there is a sequence $\{K_n^i\}_{i \in \mathbb{N}}$ such that
$$ ....M_{K^i_n}(1) \cong M_{K_n^{i-1}}(1)\cong\cdots \cong M_{K_n^1}(1) \cong M_{D'_n}(1).$$ Since
$ D_n'(1)$ is $q$-hyperbolic, by Theorem \ref{thm: doubleknotintro}, each manifold $K_n^i(1)$ is
$q$-hyperbolic, and $(b)$ follows. Furthermore, by Theorem \ref{thm: lTV bounded}, each knot $K_n^i$
is also $q$-hyperbolic, proving part $(a)$.
\end{proof}

\begin{remark}
  We note that the knots considered in Theorem \ref{thm: doubleknot2seq} are obtained by iteratively
  applying $1$-fold annulus twists. While each knot $D_n'$ admits an annulus presentation, they do
  not have monic Alexander polynomials. Indeed, for $n \in \mathbb{Z}$, we have 
  \[
   \Delta_{D_n'}(t) \doteq n t - (2n+1) + n t^{-1},
  \]
  where $\doteq$ is taken up to multiplication by $\pm t^k$. By Remark \ref{rem: goodAPMonic}, for
  $|n|>1$, the knot $D_n'$ does not admit a good annulus presentation. This means the knots
  resulting from $1$-fold annulus twists may not be distinct from $D_n'$, so the resulting sequence
  $\{K_i^n\}_{i \in \N}$ may only be a finite family of distinct $q$-hyperbolic knots.
\end{remark}

\section{An application to quantum representations}\label{sec: AMUsection}

In this section, we discuss an application to a conjecture of Andersen, Masbaum, and Ueno known as
the AMU conjecture \cite{AMU} on quantum representations of mapping class groups of surfaces.

Let $\Sigma_{g,n}$ be a compact oriented surface of genus $g$ with $n$ boundary components. 
Let Mod$(\Sigma_{g,n})$ denote its mapping class group, the group of orientation-preserving
homeomorphisms of $\Sigma_{g,n}$ fixing the boundary pointwise. 
The $SO(3)$-Witten-Reshetikhin-Turaev TQFTs \cite{ReshetikhinTuraev1991,TuraevBook} give families
of finite-dimensional projective representations of Mod$(\Sigma_{g,n})$.

Fix an odd integer $r\geq 3$, which we refer to as the \textit{level}, and let $I_r = \{0, 2, \dots,
r-3\}$ be the set of non-negative even integers less than $r-2$. Fix a primitive $2r$th root of
unity $\zeta_{2r}$ and a coloring $c$ of the components of $\partial \Sigma_{g,n}$ by elements of
$I_r$. Using the skein-theoretic framework of Blanchet, Habegger, Masbaum, and Vogel
\cite{BHMVKauffman}, this gives a finite dimensional $\mathbb{C}$-vector space $RT_r(\Sigma_{g,n},
c)$ and a respresentation 
\[
\rho_{r,c} : \text{Mod}(\Sigma_{g,n}) \rightarrow \mathbb{P}\text{Aut}(RT_r(\Sigma_{g,n}, c)),
\]
called the $SO(3)$-quantum representation of Mod$(\Sigma_{g,n})$ at level $r$. 

The Nielsen-Thurston classification implies that mapping classes $\phi \in$ Mod$(\Sigma_{g,n})$ are
either periodic, reducible, or pseudo-Anosov, and the geometry of the mapping torus $M_{\phi} =
\Sigma_{g,n}\times I / (x \sim \phi(x))$ of $\phi$ is determined by this classification. The AMU
conjecture \cite{AMU} relates the Nielsen-Thurston classification of mapping classes to their
quantum representations. 

\begin{conjecture}[AMU Conjecture, \cite{AMU}]\label{conj: AMU} Let $\phi \in$ Mod$(\Sigma_{g,n})$
  be a pseudo-Anosov mapping class. Then for any big enough level $r$, there is a choice of coloring
  $c$ of the components of $\partial \Sigma_{g,n}$ such that $\rho_{r,c}(\phi)$ has infinite order.
\end{conjecture}

\begin{remark}Note that if a mapping class $\phi \in  \rm{ Mod}(\Sigma_{g,n})$ satisfies the AMU
conjecture, then any mapping class that is a conjugate of a power of $\phi$ also satisfies the
conjecture
\end{remark}

  \begin{figure}
    \centering
    \includegraphics[scale=0.55]{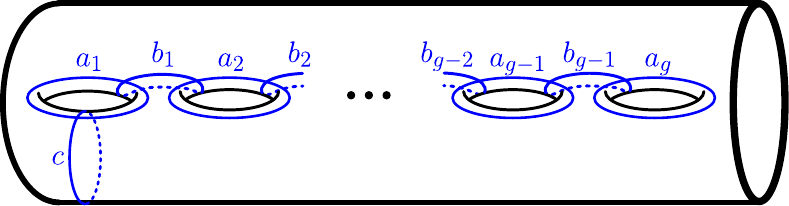}
    \caption{The curves $c, a_1, b_1, \cdots, b_{g-1}, a_g$ on $\Sigma_{g,1}$.}
    \label{fig: generatingcurves}
  \end{figure}
For a simple closed curve $a\subset  \Sigma_{g,n}$ let $\tau_a \in \rm{ Mod}(\Sigma_{g,n})$ denote
the mapping class represented by a Dehn twist along $a$ and $\tau^{-1}_a$ denote the inverse mapping
class.

On the surface of genus $g$ and with one boundary component $\Sigma_{g,1}$, consider the simple
closed curves $c, a_1, b_1, \cdots, b_{g-1}, a_g$ shown in Figure \ref{fig: generatingcurves} and
the mapping classes
$$\phi_g=\tau_c\tau_{a_1} \tau^{-1}_{b_1}\tau_{a_2} \cdots \tau^{-1}_{b_{g-1}}\tau_{a_g}, \ \ {\rm
and} \ \  \phi'_g=\tau^{-1}_c\tau^{-1}_{a_1} \tau_{b_1}\tau^{-1}_{a_2}\cdots
\tau_{b_{g-1}}\tau^{-1}_{a_g}.$$

\begin{thm} \label{AMUap} For  $g\geq 1$, the  mapping classes $\phi_g, \phi'_g\in
\rm{Mod}(\Sigma_{g,1})$ are pseudo-Anosov and they satisfy the AMU conjecture.
\end{thm}

Given $\phi\in \rm{Mod}(\Sigma_{g,1})$, the mapping torus 
$$T_{\phi}=\Sigma_{g,1}\times [-1,1]/_{(x,1)\sim ({ {\phi}(x),-1)}}$$ is a 3-manifold which fibers
over $S^1$ with fiber $\Sigma_{g,1}$ and monodromy $\phi$. By \cite[Theorem 1.2]{DKAdvances}, if
$T_{\phi}$ is $q$-hyperbolic, then ${\phi}$ satisfies the AMU conjecture. To prove Theorem
\ref{AMUap}, we will show that each of $T_{\phi_{g}}$ and $T_{\phi'_{g}}$ is homeomorphic to the
complement of a $q$-hyperbolic double twist knot.

\subsection{Fibered double twist knots} The knot $D(m, n)$ is the two-bridge knot associated with
the rational number 
\[
  \frac{n}{mn - 1} = [m, -n] = \frac{1}{m - \frac{1}{n}}. 
\]
In general, we define the continued fraction expansion (CFE) by
\[
  [a_1, a_2, \dots, a_k] := \cfrac{1}{a_1+\cfrac{1}{a_2+\cfrac{1}{a_3 + \lastcfrac{1}{a_k}}}}.
\]
We note that a CFE for a rational number is not unique, hence a two-bridge knot can have multiple
associated CFEs. The following properties of double twist knots will be useful:
\begin{enumerate} [(i)]
  \item $D(m, n) = D(n, m)$ are equivalent knots.
  \item For a double twist knot $D(m,n)$ with CFE $[a_1, \dots, a_k]$, its mirror image is $D^{*}(m,
  n): = D(-m, -n)$ and has CFE $[-a_1 \dots, -a_k]$. 
\end{enumerate}

We recall the following well known lemma that can be found, for example, in \cite{GABAIKAZEZ}.

\begin{lem}\label{lem: fiberedCFE} A two-bridge knot is fibered if and only if it has a CFE of the
  form $[a_1, \dots, a_k]$ such that $|a_i| = 2$ for $i = 1, \dots, k$ and $k$ is even.
\end{lem}

As shown in \cite{GABAIKAZEZ}, every fibered two-bridge knot can be identified with the boundary of
the Murasugi sum of a sequence of right and left Hopf bands determined by the entries in its CFE.
The monodromy of the left (resp. right) Hopf band is the left (resp. right) Dehn twist, and the
monodromy of a fibered two-bridge knot with CFE $[a_1, \dots, a_k]$ (with $|a_i| = 2$) is given by
the product of $k$ Dehn twists corresponding to each Hopf band in the Murasugi sum. In this case,
the resulting fiber is the surface of genus $\frac{k}{2}$ with one boundary component.

\begin{prop}\label{prop: twistknotsfibered} For any integer $g>0$, the double twist knot $D(3, 2g)
  \subset S^3$ is fibered with monodromy $\phi_g \in$ Mod$(\Sigma_{g, 1})$. 
\end{prop}

\begin{proof}
 
Let $n \leq -1$ be an integer, and set $g:=-n$. By the properties of twist knots, we have $D(2n, -3)
  = D^{*}(3, 2g)$. The knot $D(3, 2g)$ is the two-bridge knot associated to $[3, -2g] =
  \frac{-2g}{-6g+1}$. By Lemma \ref{lem: fiberedCFE}, $D(2n, -3)$ is fibered if and only if $D(3,
  2g)$ has a CFE of the form $[a_1, \dots, a_k]$ with $|a_i| = 2$. We will show that $D(3, 2g)$ has
  a CFE $[2, 2, -2, 2, \dots, -2, 2]$ of length $2g$.

  We note this CFE alternates sign beginning with the second term. We assume 
  \begin{align}\label{indhyp}
    \frac{-2g}{-6g + 1} &= [2, 2, -2, 2, \dots, (-1)^{2g-1} 2, (-1)^{2g} 2] = \cfrac{1}{2 + A_g},
  \end{align}
  where $A_g: = [2, -2, 2, \dots, (-1)^{2g-1} 2, (-1)^{2g} 2]$ of length $2g-1$, and proceed by
  induction. For $D(3, 2(g+1))$, we have
  \begin{align*}
    [2, 2, -2, \dots, (-1)^{2g + 2} 2] &= \cfrac{1}{2 + \cfrac{1}{2 + \cfrac{1}{-2 + A_g}}} \\
    &= \frac{-2(g+1)}{-6(g + 1) + 1} \\
    &= [3, -2(g+1)],
  \end{align*}
  where the second line follows from the identity $A_g = [3, -2g] - 2$. This establishes the claim,
  which implies that the double twist knot $D(2n, -3)$ is fibered for $n \leq -1$. 

  Following \cite{GABAIKAZEZ}, the knot $D(3, 2g)$ can be identified with the boundary of the
  Murasugi sum of $2g$ Hopf bands. The monodromy is then a product of left and right Dehn twists
  corresponding to the sign of each entry of the CFE $[2, 2, -2, 2, \dots, (-1)^{2g-1} 2, (-1)^{2g}
  2]$. These Dehn twists correspond to the collection of curves on $\Sigma_{g, 1}$ shown in Figure
  \ref{fig: generatingcurves}, and the monodromy $\phi_g=\tau_c\tau_{a_1}
  \tau^{-1}_{b_1}\tau_{a_1}\cdots \tau^{-1}_{b_{g-1}}\tau_{a_g}$.
 \end{proof}

\subsection{Proof of Theorem \ref{AMUap}} By Proposition \ref{prop: twistknotsfibered}, the knot
$D(3, 2g)$, for $g>0$, is fibered. Since these knots  are hyperbolic (see for example \cite{FG}), by
the work of Thurston the mapping class $\phi_g$  is pseudo-Anosov \cite{ThurstonGT3manifolds}.  By
Theorem \ref{thm: doubleknotintro}, these knots are $q$-hyperbolic. The mirror image $D(-2g, -3) =
D^{*}(3, 2g)$ is also hyperbolic, $q$-hyperbolic, and fibered with monodromy
$\phi'_g=\tau^{-1}_c\tau^{-1}_{a_1} \tau_{b_1}\tau^{-1}_{a_1}\cdots \tau_{b_{g-1}}\tau^{-1}_{a_1}$.
Hence, by   \cite[Theorem 1.2]{DKAdvances},  ${\phi_g}$ and $\phi'_g$ satisfy the AMU
conjecture.\qed

\medskip

\section{Low crossing knots and low volume 3-manifolds}\label{sec: low crossing tables} Tables
\ref{tab: low crossing n} and \ref{tab: low crossing m} give the twist knots $D_n = D(2n, -3)$ and
$D_n' = D(2n,-2)$ up to 10 crossings, respectively. By Lemma \ref{lem: twist knots from 4_1}, all of
these share surgeries with $4_1$.

We identify these knots by giving an alternating projection realizing the crossing numbers in
conjunction with Rolfsen's tabulation of low-crossing knots \cite{rolfsen2003knots}. We note that for
the knot $D(2n, -3)$, the resulting diagram with $2n+3$ crossings corresponding to Figure
\ref{fig:doubletwistknot} is alternating when $n \geq 1$, allowing us to identify the odd crossing
knots of Table \ref{tab: low crossing n}. To identify the even crossing knots of Table \ref{tab: low
crossing n}, we see in Figure \ref{fig: alteven} that, after applying Reidemeister moves, we obtain
an alternating diagram for $D(-2n, -3)$ with $2n+2$ crossings. 
 
\begin{figure}
  \includegraphics*[scale=0.6]{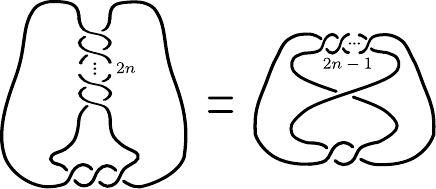}
  \caption{Alternating diagram of $D(-2n, -3)$ realizing the even crossing knots of Table \ref{tab:
  low crossing n}.}
  \label{fig: alteven}
\end{figure}

Similarly, the original diagram for $D(2n, -2)$ is also alternating with $2n+2$ crossings for $n\geq
1$, allowing us to identify the even crossing knots of Table \ref{tab: low crossing m}. Figure
\ref{fig: altodd} gives an alternating diagram for $D(-2n, -2)$ with $2n+1$ crossings, realizing the
odd crossing knots of Table \ref{tab: low crossing m}.

 \begin{table}
  \centering
\begin{center}
  \begin{tabular}{ |c|c|c|c|c|c|c|c|}
   \hline
   $n$ & -4 & -3 & -2 & -1 & 1 & 2 & 3 \\
   \hline
   $D_n$ & $10_2$ & $8_2$ & $6_2$ & $4_1$ & $5_2$ & $7_3$ & $9_3$ \\
   \hline
  
  \end{tabular}
   \vskip 0.06in \caption{Low-crossing knots $D_n = D(2n, -3)$.}
  \label{tab: low crossing n}
  \end{center}
\end{table}

\begin{figure}[h]
  \includegraphics*[scale=0.7]{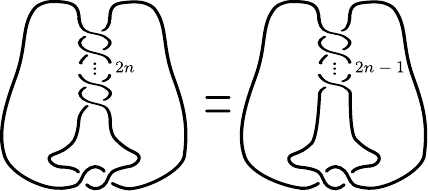}
  \caption{Alternating diagram of $D(-2n, -2)$ realizing the odd crossing knots of Table \ref{tab:
  low crossing m}.}
  \label{fig: altodd}
 \end{figure}

\begin{table}
  \centering
  \begin{center}
    \begin{tabular}{ |c|c|c|c|c|c|c|c|c| }
     \hline
     $n$ & -4 & -3 & -2 & -1 & 1 & 2 & 3 & 4 \\
     \hline
     $D_n'$ & $9_2$ & $7_2$ & $5_2$ & $3_1$ & $4_1$ & $6_1$ & $8_1$ & $10_1$\\
     \hline
   
    \end{tabular}
    \vskip 0.06in \caption{Low-crossing knots $D_n' = D(2n, -2)$.}
    \label{tab: low crossing m}
    \end{center}
  \end{table}

By Proposition \ref{prop: twistknotsfibered}, for $n \leq -1$ the knot $D(2n, -3)$ is fibered with
genus $|n|$. By Table \ref{tab: low crossing n}, for $n=-4,-3,-2,-1$, the knot $D(2n, -3)$ is
identified as the corresponding knot shown in the table. Indeed the knots $10_2, 8_2, 6_2,$ and
$4_1$ are well known to be fibered of genus, $4,3,2,$ and $1$, respectively \cite{knotinfo}.

\begin{remark}
  The manifold $M_{4_1}(-5)$, which is homeomorphic to $M_{5_2}(5)$ by Lemma \ref{lem: twist knots
  from 4_1} and Table \ref{tab: low crossing n}, is known as the Meyerhoff manifold. It is the
  second-smallest volume closed orientable hyperbolic $3$-manifold, with volume approximately
  $0.9814$. 
\end{remark}

Futer, Purcell, and Schleimer recently wrote a software package \cite{FuterPurcellSchleimerCode}, in
conjunction with forthcoming work \cite{FuterPurcellSchleimerPaper}, for testing the cosmetic
surgery conjecture. At our request, they extended the code to allow for testing whether pairs of
cusped $3$-manifolds have common Dehn fillings as well as identifying those fillings. Running the
code for knots up to 12 crossings, as well as on SnapPy's census of the 1267 hyperbolic knot
complements that can be triangulated with fewer than 10 tetrahedra \cite{SnapPy}, they verified the
data given in Tables \ref{tab: low crossing n} and \ref{tab: low crossing m} and identified many
additional knots which share surgeries with $4_1$.

  Tables \ref{tab: qhyp knots2-7} and \ref{tab: qhyp knots8-9}  list all the knot complements from
the SnapPy census of hyperbolic cusped 3-manifolds that admit triangulations with at most nine
tetrahedra and have shared Dehn fillings with the complement of $4_1$. The information on the tables
is recorded as follows: Column 1 presents the knot $K$  with the notation used in the SnapPy census
while Column 2 gives the approximate volume of $M_K$. Column 3 gives the surgery slopes  $a/b$,
$p/q$, with $M_K(a/b)\cong M_{4_1}(p/q)$ and Column 4 gives the approximate volume of that manifold.
All of these knots, many of which are twisted torus knots, have known diagrams, but they may be
complicated and require hundreds of crossings. Using the tables of
\cite{CallahanDeanWeeks,ChampanerkarKofmanPatterson,ChampanerkarKofmanMullen}, we may identify some
of the examples from the knot tables.  
Note that since since $4_1$ is amphicheiral, we have $M_{4_1}(-p/q)\cong M_{4_1}(p/q)$. Hence the
  slopes $p/q$ in the tables can be, equivalently, be replaced with its negative.

\begin{table}
   \begin{center}
    \begin{tabular}{ |c|c|c|c|c|}
     \hline
     $K$ & vol$(M_K)$ & Slopes $a/b$, $p/q$ & vol$(M_K(a/b))$ & Knot \\
     \hline
     $K2_1$      & 2.029883 & -               & - &  $4_1$ \\ \hline
     $K3_2$      & 2.828122 & 5, -5            & 0.981369 &  $5_2$ \\
                 &          & 1, $1/2$        & 1.398509 & \\ \hline
     $K4_1$      & 3.163963 & 1, $-1/2$       & 1.398509 &   $6_1$  \\ \hline
     $K4_2$      & 3.331744 & 1, $1/3$        & 1.731983 &   $7_2$  \\ \hline
     $K5_2$      & 3.427205 & 1, $-1/3$       & 1.731983 &  $8_1$ \\ \hline
     $K5_3$      & 3.486660 & 1, $1/4$        & 1.858138 &  $9_2$ \\ \hline
     $K5_9$      & 4.056860 & $-2$, $2/3$     & 1.737124 &  $10_{132}$ \\ \hline
     $K5_{12}$   & 4.124903 & 3, $3/2$        & 1.440699 &  $8_{20}$ \\ \hline
     $K5_{13}$   & 4.124903 & 1, $1/3$        & 1.731983 & $11n_{38}$  \\ \hline
     $K5_{19}$   & 4.400833 & $-7$, $-7/2$    & 1.649610 & $6_2$ \\ \hline
     $K5_{20}$   & 4.592126 & 9, $-9/2$       & 1.752092 &  $7_3$ \\ \hline
     $K6_1$      & 3.526196 & 1, $-1/4$       & 1.858138 & $10_1$  \\ \hline
     $K6_2$      & 3.553820 & 1, $1/5$        & 1.918602 &  $11a_{247}$ \\ \hline
     $K6_8$      & 4.293750 & $-3$, $3/5$     & 1.921026 &   \\ \hline
     $K6_9$      & 4.307917 & $-2$, $2/5$     & 1.919520 &   \\ \hline
     $K6_{23}$   & 4.935243 & $-11$, $-11/3$  & 1.876053 &  $8_2$ \\ \hline
     $K6_{24}$   & 4.994856 & $13$, $-13/3$    & 1.903695 & $9_3$  \\ \hline
     $K6_{37}$   & 5.413307 & 7, $7/3$        & 1.805827 &  $15n_{41127}$ \\ \hline
     $K7_1$      & 3.573883 & 1, $-1/5$       & 1.918602 &  $12a_{803}$ \\ \hline
     $K7_2$      & 3.588914 & 1, $1/6$        & 1.952062 & $13a_{3143}$ \\ \hline
     $K7_{10}$   & 4.354670 & $-4$, $4/7$     & 1.973762 &   \\ \hline
     $K7_{11}$   & 4.359783 & $-3$, $3/7$     & 1.973161 &   \\ \hline
     $K7_{41}$   & 4.933530 & $-5$, $5/4$     & 1.873482 &   \\ \hline
     $K7_{44}$   & 4.993457 & 7, $7/5$        & 1.932061 &   \\ \hline
     $K7_{45}$   & 5.114841 & $-15$, $-15/4$  & 1.946574 &  $10_2$ \\ \hline
     $K7_{46}$   & 5.140207 & $17$, $-17/4$   & 1.957888 &  $11a_{364}$ \\ \hline
     $K7_{95}$   & 5.860539 & 11, $11/2$      & 1.822675 &  $10_{128}$ \\ \hline
     $K7_{96}$   & 5.860539 & 13, $13/3$      & 1.903695 &  $11n_{57}$ \\ \hline
     $K7_{98}$   & 5.904086 & 14, $14/3$      & 1.915331 &  $12n_{243}$ \\ \hline
     $K7_{129}$  & 6.922634 & $-7$, $7/3$     & 1.805827 &   \\
      \hline
    \end{tabular}
     \vskip 0.06in \caption{Knots in the SnapPy census of cusped hyperbolic 3-manifolds that share
    surgeries with $4_1$. } 
    \label{tab: qhyp knots2-7}
    \end{center}
  \end{table}

   The following applies to all the knots in Tables \ref{tab: qhyp knots2-7} and \ref{tab: qhyp
  knots8-9}:
  
  \begin{prop}\label{prop: bounds} Suppose that $K$ is a hyperbolic knot  in $S^3$ such that $M_K$
  admits a triangulation with $t$ tetrahedra. Suppose, moreover, that  $M_K(a/b)\cong M_{4_1}(p/q)$,
  for some slopes $a/b, p/q \in \Q$, where $p/q$ is a non-exceptional slope of $4_1$. Then we have
  
  $$ {\rm vol}(M_K(a/b))\leq \textit{lTV}(M_{K})\leq v_{\rm {oct}}\cdot t,$$
  
  where $v_{\rm {oct}}\approx 3.6638$ is the volume of the ideal regular octahedron.
  \end{prop}
\begin{proof} The upper bound follows at once from \cite[Corollary 3.9]{growth6j}.

By Theorem \ref{thm: Ohtsuki, Wong-Yang q-hyp}, ${\rm vol}(M_{4_1}(p/q))=\textit{lTV}(M_{4_1}(p/q))$
and, by assumption, $M_K(a/b)\cong M_{4_1}(p/q)$. Combining these with Theorem \ref{thm: lTV
bounded}, leads to the lower bound of $\textit{lTV}(M_{K})$.
\end{proof}

\begin{remark}
  It is known that the lowest volume hyperbolic knots are $4_1$, $5_2$, $6_1$, and the
  $(-2,3,7)$-pretzel knot. In \cite{chenyang2018vol}, Chen and Yang give computational evidence for
  the Turaev--Viro invariant volume conjecture for each of these knots. By Lemma \ref{lem: twist
  knots from 4_1}, $5_2$ and $6_1$ share surgeries with $4_1$, hence are $q$-hyperbolic by Theorem
  \ref{thm: doubleknotintro}. However, the $(-2,3,7)$-pretzel knot was shown not to share any
  surgeries with $4_1$ using the code of Futer--Purcell--Schleimer \cite{FuterPurcellSchleimerCode}.
  Similarly the knot $6_3$ was shown not to share any surgeries with $4_1$, making it the only
  hyperbolic knot with up to six crossings for which $q$-hyperbolicity cannot be decided with the
  methods of this paper.
\end{remark}

\begin{remark}
  Table \ref{tab: qhyp knots2-7} includes the knots $8_{20}, 10_{132}, 11n_{38}$, and $11n_{57}$.
  According to KnotInfo \cite{knotinfo}, the complements $M_{8_{20}}$, $M_{10_{132}}$,
  $M_{11n_{38}}$, and $M_{11n_{57}}$ are also fibered, so their associated monodromies (as well as
  powers of conjugates of those mapping classes) satisfy he AMU
  conjecture. 
\end{remark}

\begin{remark}
  As we see in Table \ref{tab: qhyp knots2-7}, the knot $K5_{12} = 8_{20}$ shares a surgery with
  $4_1$. In particular, $M_{4_1}(3/2) \cong M_{8_{20}}(3)$. This is a closed manifold of volume
  $\approx 1.440699$. In addition, as shown in \cite{AJLO2015}, $8_{20}$ admits a good annulus
  presentation. This means an analogous version of Theorem \ref{thm: 6_2annulustwist} also holds for
  $8_{20}$.
\end{remark}

  \begin{table}
       \begin{center}
      \begin{tabular}{ |c|c|c|c|c| }
       \hline
       $K$ & vol$(M_K)$ & Slopes $a/b$, $p/q$ & vol$(M_K(a/b))$ & Knot \\
       \hline
       $K8_1$      & 3.60046726278 & 1, $-1/6$      & 1.9520620754135 &  $14a_{12741}$ \\ \hline
       $K8_2$      & 3.6095391745 & 1, $1/7$        & 1.9724601973306 &  $15a_{54894}$ \\ \hline
       $K8_9$      & 4.3790606712 & $-5$, $5/9$     & 1.9957717794010 &   \\ \hline
       $K8_{10}$   & 4.38145643736 & $-4$, $4/9$    & 1.9954776244141 &   \\ \hline
       $K8_{61}$   & 5.07001608898 & 9, $9/7$       & 1.9788631982608 &   \\ \hline
       $K8_{62}$   & 5.0827080657 & 11, $11/8$      & 1.9914466741922 &   \\ \hline
       $K8_{64}$   & 5.1955903246 & $-19$, $19/5$   & 1.9776430099735 &  $12a_{722}$ \\ \hline
       $K8_{65}$   & 5.2086109485 & $-21$, $21/5$   & 1.983357467405  &  $13a_{4874}$ \\ \hline
       $K8_{96}$   & 5.75222662008 & 11, $11/5$     & 1.9478817102192 &   \\ \hline
       $K8_{105}$  & 5.8281487245 & $-16$, $16/7$   & 1.9891579197851 &   \\ \hline
       $K8_{133}$  & 6.1411744018 & 22, $22/5$      & 1.9859441335531 &   \\ \hline
       $K8_{135}$  & 6.1504206159 & 23, $23/5$      & 1.9883610027459 &  $T(7, 9, 6,-6, 5,-1)$ \\ \hline
       $K8_{143}$  & 6.2597017011 & $-13$, $13/4$   & 1.9334036965515 &   \\ \hline
       $K8_{145}$  & 6.27237250941 & 1, $1/2$       & 1.3985088841508 &  $14n_{18212}$  \\ \hline
       $K8_{268}$  & 7.26711903086 & 9, $9/4$       & 1.9026876676640 &   \\ \hline
       $K9_1$      & 3.61679304740 & $-1$, $-1/7$   & 1.9724601973306 &   \\ \hline
       $K9_2$      & 3.62268440821 & 1, $1/8$       & 1.9857927453641 &   \\ \hline
       $K9_8$      & 4.3912243457 & $-6$, $6/11$    & 2.0069885249369 &   \\ \hline
       $K9_9$      & 4.39253386353 & 5, $5/11$      & 2.0068241855029 &   \\ \hline
       $K9_{83}$   & 5.1043901461 & 13, $13/10$     & 2.004926648441  &   \\ \hline
       $K9_{85}$   & 5.1089909300 & $-15$, $15/11$  & 2.0095023855854 &   \\ \hline
       $K9_{93}$   & 5.23864536794 & 23, $-23/6$    & 1.9940644235057 &  $14a_{12197}$ \\ \hline
       $K9_{94}$   & 5.24618858374 & 25, $-25/6$    & 1.9973474789782 &  $15a_{85258}$ \\ \hline
       $K9_{152}$  & 5.8653629974 & 20, $20/9$      & 2.004886373798  &   \\ \hline
       $K9_{155}$  & 5.8812168764 & $-25$, $25/11$  & 2.0133867882020 &   \\ \hline
       $K9_{242}$  & 6.2152290434 & 31, $31/7$      & 2.007727892627  &   \\ \hline
       $K9_{244}$  & 6.21858163948 & $-32$, $32/7$  & 2.0085996110216 &   \\ \hline
       $K9_{282}$  & 6.5328202770 & $-21$, $21/4$   & 1.9754820965797 &   \\ \hline
       $K9_{296}$  & 6.6272713527 & 27, $27/5$      & 1.9965186652378 &   \\ \hline
       $K9_{299}$  & 6.6445653099 & 19, $19/3$      & 1.9565702867106 &   \\ \hline
       $K9_{435}$  & 7.2356793751 & $-3$, $3/4$     & 1.8634426716184 &   \\ \hline
        \hline
      \end{tabular}
       \vskip 0.06in
      \caption{ Knots in the
   SnapPy census of cusped hyperbolic 3-manifolds that share surgeries with $4_1$.} 
      \label{tab: qhyp knots8-9}
      \end{center}
    \end{table}


\bibliographystyle{abbrv}

\bibliography{biblio}

\begin{thebibliography}{10}

\bibitem{AJLO2015}
T.~Abe, I.~D. Jong, J.~Luecke, and J.~Osoinach.
\newblock Infinitely many knots admitting the same integer surgery and a
  four-dimensional extension.
\newblock {\em Int. Math. Res. Not. IMRN}, 2015(22):11667--11693, 2015.

\bibitem{AJOT2013}
T.~Abe, I.~D. Jong, Y.~Omae, and M.~Takeuchi.
\newblock Annulus twist and diffeomorphic 4-manifolds.
\newblock {\em Math. Proc. Cambridge Philos. Soc.}, 155(2):219--235, 2013.

\bibitem{AbeTagami}
T.~Abe and K.~Tagami.
\newblock Knots with infinitely many non-characterizing slopes.
\newblock {\em Kodai Math. J.}, 44(3):395--421, 2021.

\bibitem{AMU}
J.~E. Andersen, G.~Masbaum, and K.~Ueno.
\newblock {Topological quantum field theory and the Nielsen-Thurston
  classification of M(0,4)}.
\newblock {\em {Math Proc. Cambridge Philos. Soc.}}, 141:{477--488}, {2006}.

\bibitem{growth6j}
G.~Belletti, R.~Detcherry, E.~Kalfagianni, and T.~Yang.
\newblock Growth of quantum {$6j$}-symbols and applications to the volume
  conjecture.
\newblock {\em J. Differential Geom.}, 120(2):199--229, 2022.

\bibitem{TVCompact}
R.~Benedetti and C.~Petronio.
\newblock On {R}oberts' proof of the {T}uraev-{W}alker theorem.
\newblock {\em J. Knot Theory Ramifications}, 5(4):427--439, 1996.

\bibitem{BHMVKauffman}
C.~Blanchet, N.~Habegger, G.~Masbaum, and P.~Vogel.
\newblock Topological quantum field theories derived from the {K}auffman
  bracket.
\newblock {\em Topology}, 34(4):883--927, 1995.

\bibitem{CallahanDeanWeeks}
P.~J. Callahan, J.~C. Dean, and J.~R. Weeks.
\newblock The simplest hyperbolic knots.
\newblock {\em J. Knot Theory Ramifications}, 8(3):279--297, 1999.

\bibitem{ChampanerkarKofmanMullen}
A.~Champanerkar, I.~Kofman, and T.~Mullen.
\newblock The 500 simplest hyperbolic knots.
\newblock {\em J. Knot Theory Ramifications}, 23(12):1450055, 34, 2014.

\bibitem{ChampanerkarKofmanPatterson}
A.~Champanerkar, I.~Kofman, and E.~Patterson.
\newblock The next simplest hyperbolic knots.
\newblock {\em J. Knot Theory Ramifications}, 13(7):965--987, 2004.

\bibitem{chenyang2018vol}
Q.~Chen and T.~Yang.
\newblock Volume conjectures for the {R}eshetikhin-{T}uraev and the
  {T}uraev-{V}iro invariants.
\newblock {\em Quantum Topol.}, 9(3):419--460, 2018.

\bibitem{SnapPy}
M.~Culler, N.~M. Dunfield, M.~Goerner, and J.~R. Weeks.
\newblock Snap{P}y, a computer program for studying the geometry and topology
  of $3$-manifolds.
\newblock Available at \url{http://snappy.computop.org} (DD/MM/YYYY).

\bibitem{2-cable}
R.~Detcherry.
\newblock Growth of {T}uraev-{V}iro invariants and cabling.
\newblock {\em J. Knot Theory Ramifications}, 28(14):1950041, 8, 2019.

\bibitem{DKAdvances}
R.~Detcherry and E.~Kalfagianni.
\newblock Quantum representations and monodromies of fibered links.
\newblock {\em Adv. Math.}, 351:676--701, 2019.

\bibitem{detcherryKalfagianni2019gromov}
R.~Detcherry and E.~Kalfagianni.
\newblock Gromov norm and {T}uraev-{V}iro invariants of 3-manifolds.
\newblock {\em Ann. Sci. \'{E}c. Norm. Sup\'{e}r. (4)}, 53(6):1363--1391, 2020.

\bibitem{DKIndiana}
R.~Detcherry and E.~Kalfagianni.
\newblock Cosets of monodromies and quantum representations.
\newblock {\em Indiana Univ. Math. J.}, 71(3):1101--1129, 2022.

\bibitem{colJvolDKY}
R.~Detcherry, E.~Kalfagianni, and T.~Yang.
\newblock Turaev-{V}iro invariants, colored {J}ones polynomials, and volume.
\newblock {\em Quantum Topol.}, 9(4):775--813, 2018.

\bibitem{FG}
D.~Futer and F.~Gu\'{e}ritaud.
\newblock Angled decompositions of arborescent link complements.
\newblock {\em Proc. Lond. Math. Soc. (3)}, 98(2):325--364, 2009.

\bibitem{FuterPurcellSchleimerCode}
D.~Futer, J.~S. Purcell, and S.~Schleimer.
\newblock Excluding cosmetic surgeries using hyperbolic geometry.
\newblock Code available at \url{http://github.com/saulsch/Cosmetic.git}.

\bibitem{FuterPurcellSchleimerPaper}
D.~Futer, J.~S. Purcell, and S.~Schleimer.
\newblock Excluding cosmetic surgeries on hyperbolic 3-manifolds, 2024.

\bibitem{GABAIKAZEZ}
D.~Gabai and W.~H. Kazez.
\newblock Pseudo-anosov maps and surgery on fibred 2-bridge knots.
\newblock {\em Topology and its Applications}, 37(1):93--100, 1990.

\bibitem{knotinfo}
C.~Livingston and A.~H. Moore.
\newblock Knotinfo: Table of knot invariants.
\newblock URL: \url{knotinfo.math.indiana.edu}, March 2023.

\bibitem{OhtsukiFig8}
T.~Ohtsuki.
\newblock On the asymptotic expansion of the quantum su(2) invariant at q
  =exp(4$\pi$/n) for closed hyperbolic 3–manifolds obtained by integral
  surgery along the figure-eight knot.
\newblock {\em Algebraic \& Geometric Topology}, 2018.

\bibitem{ReshetikhinTuraev1991}
N.~Reshetikhin and V.~G. Turaev.
\newblock Invariants of {$3$}-manifolds via link polynomials and quantum
  groups.
\newblock {\em Invent. Math.}, 103(3):547--597, 1991.

\bibitem{RobertsSkein}
J.~Roberts.
\newblock Skein theory and {T}uraev-{V}iro invariants.
\newblock {\em Topology}, 34(4):771--787, 1995.

\bibitem{Rolfsen}
D.~Rolfsen.
\newblock {\em Knots and links}.
\newblock Mathematics Lecture Series, No. 7. Publish or Perish, Inc., Berkeley,
  Calif., 1976.

\bibitem{rolfsen2003knots}
D.~Rolfsen.
\newblock {\em Knots and Links}.
\newblock AMS Chelsea Publishing Series. AMS Chelsea Pub., 2003.

\bibitem{ThurstonGT3manifolds}
W.~P. Thurston.
\newblock {\em The geometry and topology of three-manifolds}.
\newblock Princeton University Math Department Notes, 1979.

\bibitem{TuraevBook}
V.~G. Turaev.
\newblock {\em Quantum invariants of knots and 3-manifolds}, volume~18 of {\em
  De Gruyter Studies in Mathematics}.
\newblock Walter de Gruyter \& Co., Berlin, 1994.

\bibitem{wongYangF8}
K.~H. Wong and T.~Yang.
\newblock On the volume conjecture for hyperbolic {D}ehn-filled $3$-manifolds
  along the figure-eight knot, 2022.

\end{thebibliography}

\end{document}